\documentclass[a4paper,final]{amsart}  % `article' with A4 paper size option

\usepackage{amsmath}
\usepackage{amsthm}
\usepackage{amssymb}
\usepackage{xypic}
\usepackage{graphicx}
\usepackage{psfrag}
\usepackage{mdwlist}
\usepackage{wrapfig}

\theoremstyle{plain}
\newtheorem{theorem}{Theorem}

\newtheorem{lemma}[theorem]{Lemma}

\theoremstyle{definition}
\newtheorem{remark}[theorem]{Remark}
\newtheorem{example}[theorem]{Example}

\newcommand{\IP}{{\mathbb{P}}}
\newcommand{\IQ}{{\mathbb{Q}}}

\newcommand{\IR}{{\mathbb{R}}}
\newcommand{\IZ}{{\mathbb{Z}}}

\newcommand{\ko}{{\mathcal O}}

\newcommand{\kc}{{\mathcal C}}

\newcommand{\rank}{\text{rk}}
\newcommand{\supp}{\text{supp}}

\newcommand{\kk}{{\mathbf{k}}}

\newcommand{\inv}{^{-1}}

\DeclareMathOperator{\Pic}{\mathrm{Pic}}
\DeclareMathOperator{\Hom}{\mathrm{Hom}}
\DeclareMathOperator{\Aut}{\mathrm{Aut}}

\DeclareMathOperator{\Coh}{\mathrm{Coh}}

\newcommand{\FM}{\mathsf{FM}}
\newcommand{\TTT}{\mathsf{T}\!}   % special case for the \TTT_X notation

\newcommand{\LLL}{\mathsf{L}}
\newcommand{\RRR}{\mathsf{R}}
\newcommand{\DDD}{\mathcal{D}}

\newcommand{\isom}{ \text{{\hspace{0.48em}\raisebox{0.8ex}{${\scriptscriptstyle\sim}$}}}
                    \hspace{-0.65em}{\rightarrow}\hspace{0.3em}} % Gewalt, ohne Trennung

\newcommand{\arrd}{ \ar@{-}[r] \ar@{=}[d] }

\begin{document}

\begin{center}
{\Large\bfseries\sffamily Autoequivalences of toric surfaces}

\medskip

Nathan Broomhead, David Ploog\renewcommand{\thefootnote}{\fnsymbol{footnote}}
\footnote{
The second author was supported by DFG priority program 1388 ``Darstellungstheorie".


}
\setcounter{footnote}{0}
\end{center}

\bigskip

%\begin{abstract} % does not work with the jdg style file, using quote instead
\begin{quote}
\small \noindent 
We show that the autoequivalence group of the derived category of any smooth projective toric surface
is generated by the standard equivalences and spherical twists obtained from $-2$-curves. In many
cases we give all relations between these generators.
We also prove a close link between spherical objects and certain pairs of exceptional objects.
\end{quote}
%\end{abstract}

\bigskip

\noindent
In this article, we study the derived category of any smooth, projective, toric surface
or rather its group of autoequivalences. We give generators for each such group and in
great, although not full, generality we are able to go further and write down an explicit
description of the group; see Theorem~\ref{thm:AutDX}. Similar descriptions are only
known for abelian varieties \cite{Orlov} and for varieties with ample or anti-ample 
canonical bundle \cite{BO}.
 
We follow the philosophy established in the work of Bondal, Mukai, Orlov and others
that the group of autoequivalences of a variety is highly influenced by the positivity of
the (anti-)canonical bundle.
In particular, varieties with trivial canonical bundle possess the richest autoequivalences,
while the group of autoequivalences for varieties at both ends of the spectrum (i.e.\ $K_X$
ample or anti-ample) are minimal by the famous result of Bondal and Orlov \cite{BO}.

Our surfaces have rather negative canonical bundle as $-K_X$ is big. Thus we expect rather few 
autoequivalences beyond the standard ones. However, toric surfaces can contain smooth rational 
curves of self-intersection $-2$, a simple example being the second Hirzebruch surface, and 
such curves give rise to spherical twists. We prove that these twists are essentially the only
new autoequivalences which can occur: Theorem~\ref{thm:AutDX} is the general result and
Theorem~\ref{thm:toric} gives the application to toric surfaces.
At the opposite end of the surface classification, Ishii and Uehara \cite[Theorem 1.3]{IU} 
already proved a corresponding statement: the only non-standard autoequivalences for smooth 
projective surfaces of general type whose canonical model has at most $A_n$-singularities come 
from spherical twists associated to $-2$-curves.
We draw heavily upon their results in this article. In many cases however, we go further, and
describe all relations between the generators.

In Theorem~\ref{thm:sphexc}, we prove a relationship between exceptional and spherical 
objects on a smooth projective toric surface. It is well-known that such surfaces come with
an abundance of exceptional objects, including, for example, all line bundles.
We link the rather few spherical objects to the wealth of exceptional ones, by discussing
exceptional presentations of spherical objects, i.e.\ exact triangles $E'\to E\to S$ with
$E'$, $E$ exceptional and $S$ spherical.

While toric surfaces form the main class of examples, our results actually hold in greater 
generality. In Section~\ref{sec:complexity-one}, we give some examples of non-toric rational
surfaces where the group of autoequivalences can also be described.
\bigskip

\noindent
\textbf{Acknowledgements:} We would like to thank those participants of the Warwick August 2010
conference on orbifolds and the McKay correspondence who discussed the topic with us. 
We also thank Christian B\"ohning, Lutz Hille, Akira Ishii, Arthur Prendergast-Smith, and
Hokuto Uehara for answering our questions. We are very grateful to Fabrizio Catanese for help
with Lemma~\ref{lem:zariski}, Anne Fr\"uhbis-Kr\"uger for help with the calculation of
Example~\ref{computerexample} and to Hendrik S\"u\ss\ who provided much input into
Section~\ref{sec:complexity-one}.
We thank the referee and the editor for useful comments.

\section{Setup and Results}
\label{sec:setup}

\noindent
Let $X$ be a smooth, projective surface over an algebraically closed field $\kk$. Denote its 
derived category by $\DDD(X):=D^b(\Coh(X))$, this is a $\kk$-linear, triangulated category. 
See the textbook \cite{HuyFM} for background on derived categories of varieties.
For any two objects $A,B\in\DDD(X)$, we set $\Hom^\bullet(A,B)=\bigoplus_i\Hom(A,B[i])[-i]$;
this is a complex of $\kk$-vector spaces with trivial differential. Note that by our 
assumptions on $X$, the dimension of $\Hom^\bullet(A,B)$ is finite. Let $\omega_X$ denote the
canonical bundle on $X$, then $-\otimes\omega_X[2]\colon\DDD(X)\isom\DDD(X)$ is a Serre functor,
i.e.\ there are canonical isomorphisms, bifunctorially in $A,B\in\DDD(X)$
\[ \Hom(A,B)\cong\Hom(B,A\otimes\omega_X[2])^* . \] 

The \emph{standard autoequivalences} form the subgroup of $\Aut(\DDD(X))$
\[ A(X) := (\Pic(X) \rtimes \Aut(X)) \times \IZ[1] \]
where $\Pic(X)$ are the line bundle twists, $\Aut(X)$ surface automorphisms and $\IZ[1]$
the shifts of complexes.

Sometimes, $\Aut(\DDD(X))$ is strictly larger than $A(X)$. For example, when $X$ is an abelian
surface there will always be the non-standard original Fourier-Mukai transform (see \cite{mukai}).
Another source for non-standard equivalences are the \emph{spherical twists} $\TTT_S$
introduced in \cite{ST}; we refer to \cite[\S8]{HuyFM} for a concise presentation.
These are built from \emph{spherical objects} in $\DDD(X)$, i.e.\ objects $S\in\DDD(X)$ such that
$\Hom^\bullet(S,S)=\kk\oplus\kk[-2]$ and $S\otimes\omega_X\cong S$. A crucial example is given by
$S=\ko_C$, where $C\subset X$ is a smooth, rational curve with self-intersection number $C^2=-2$.

Let us introduce some notation:
\begin{align*}
 \Delta(X)      &:= \{ C\subset X \text{ irreducible $-2$-curve} \} , && \text{a possibly infinite set}; \\
 \Pic_\Delta(X) &:= \langle\ko_X(C) \mid C\in\Delta\rangle , && \text{as a subgroup of } \Pic(X); \\
 B(X)           &:= \langle \TTT_S \mid S\in\DDD(X) \text{ spherical} \rangle , && \text{a normal subgroup of } \Aut(\DDD(X)) .
\end{align*}

In \cite{IU}, Ishii and Uehara prove that for a smooth projective surface of general type whose
canonical model has at worst $A_n$-singularities, the autoequivalences are generated by $B(X)$ 
and the standard autoequivalences. The following theorem is a counterpart to this in the case where $-K_X$ is big, i.e.\ a sufficiently
high power of $-K_X$ gives a birational map from $X$ to a surface in projective space (see 
\cite[Definition 2.2.1]{Lazarsfeld}). Under
certain conditions, we can go further and describe the structure of the group of autoequivalences.

\begin{theorem} \label{thm:AutDX}
Let $X$ be a smooth, projective surface and consider the conditions

\noindent
\begin{tabular}{l @{\hspace{0.8em}} p{0.92\textwidth}}
(1) & The anti-canonical bundle is big. \\
(2) & The $-2$-curves on $X$ form disjoint chains of type $A$. \\
(3) & $\Pic(X)\cong\Pic_\Delta(X)\oplus P$ where $P$ is an $\Aut(X)$-invariant complement.
\end{tabular}

\noindent
If $X$ satisfies (1) and (2) then $\Aut(\DDD(X))$ is generated by $\Pic(X)$, $\Aut(X)$, 
$\IZ[1]$ and $B(X)$. If $X$ satisfies (1)--(3) then there is the following decomposition of
$\Aut(\DDD(X))$
\[ \Aut(\DDD(X)) = B(X) \rtimes (P \rtimes \Aut(X)) \times \IZ[1] . \]
\end{theorem}

The conditions (1)--(3) are satisfied by broad classes of surfaces as the next two results
show. To state Theorem~\ref{thm:toric}, we need to introduce one further piece of terminology
(see Section~\ref{sec:toric} for details): $\Aut(\Sigma(X))$ is the group of automorphisms of
a fan $\Sigma$ giving the toric surface $X$; this is a finite subgroup of $\Aut(X)$.
% It is obvious that $\Aut(\Sigma(X))$ does not depend on the choice of $\Sigma$.

\begin{theorem} \label{thm:toric}
If $X$ is a smooth, projective, toric surface, then the conditions (1) and (2)
of Theorem~\ref{thm:AutDX} are satisfied. All but three such surfaces admit a 
splitting of $\Pic_\Delta(X)\subset\Pic(X)$. An $\Aut(X)$-invariant complement 
exists if and only if an $\Aut(\Sigma(X))$-invariant complement exists.
\end{theorem}

An $\Aut(\Sigma(X))$-invariant complement tautologically exists whenever there 
are no non-trivial fan automorphisms (the `generic' case), yielding the full
structure of $\Aut(\DDD(X))$ from Theorem~\ref{thm:AutDX}. When $\Aut(\Sigma(X))$
is non-trivial, an invariant complement may or may not exist; in
Section~\ref{sec:toric} we give examples of both possibilities.

Theorem~\ref{thm:AutDX} also applies to some non-toric surfaces.
The result below is proved Section~\ref{sec:complexity-one}, where we also give
examples of such surfaces meeting condition (3).

\begin{theorem} \label{thm:complexity-one}
If $X$ is a smooth, projective, rational surface with a $\kk^*$-action such that
all isotropy groups are either 0 or $\kk^*$, then the conditions (1) and (2) of 
Theorem~\ref{thm:AutDX} are satisfied.
% Condition (4) holds if there are two transversal pointwise fixed curves of 
% self-intersection different from $-2$.
\end{theorem}

We end this section with a couple of remarks about Theorem~\ref{thm:AutDX}.
\begin{remark}
On any smooth projective surface $X$ with $K_X\neq0$, spherical objects are necessarily supported
on curves. The relations  $\Pic(X)\cap B(X) = \Pic_\Delta(X)$ and $\Aut(X)\cap B(X) = 1$
then hold \cite[\S4]{IU}.
Generally, $B(X)$ is a normal subgroup of $\Aut(\DDD(X))$.
These properties hint at the semi-direct decomposition of $\Aut(\DDD(X))$
in Theorem~\ref{thm:AutDX}, but there are two obstacles: our choice of $B(X)$ as the
normal factor of $\Aut(\DDD(X))$, together with $\Pic_\Delta(X)\subset B(X)$, demands a splitting
$\Pic(X)=P\oplus\Pic_\Delta(X)$. Next, the action of $\Aut(X)$ on $\Pic(X)$
forces $P$ to be $\Aut(X)$-invariant. These two conditions are exactly the content of (3).
\end{remark}

%\begin{remark}
%The assumptions of the theorem are probably not optimal. We have to state (1) because we rely
%on \cite{IU}. However, it is expected that this result extends to types beyond the $A_n$-case.
%
%We state (2) in order to be able to use \cite{Kawamata-DK}. While it is obvious that Kawamata's
%result will not extend to arbitrary surfaces --- the anti-canonical bundle of $\IP^2$ blown up
%in nine general points is not big --- the conclusion of Theorem~\ref{thm:AutDX} is expected to
%hold for all smooth, projective surfaces with non-zero canonical class.
%
%In the spirit of our introduction, we would like to point out that $-K_X$ big implies
%$\text{FM}(X)=\{X\}$, i.e.\ there are no non-trivial Fourier-Mukai partners, as follows again
%from \cite{Kawamata-DK}. Note that this is no longer true in dimension 3. However, our result
%shows that many surfaces have $\Aut(\DDD(X))\neq A(X)$ despite $\text{FM}(X)=\{X\}$.
%\end{remark}

\begin{remark} \label{rmk:splitting_braid_group}
We state what is known about $B(X)$.
If $X$ is rational with $-K_X$ big, as it will be in the examples of Sections~\ref{sec:toric} and
\ref{sec:complexity-one}, then $\Delta$ is a finite set as follows, for example, from the fact that
$X$ is a Mori dream space \cite[\S2]{TVV}.

Let $\kc=\bigcup_{C\in\Delta}C$ be the union of all $-2$-curves on $X$ and 
$\kc=\kc_1\sqcup\dots\sqcup\kc_r$ be its decomposition into connected components. Let 
$B(X)|_{\kc_l}\subset B(X)$ be the subgroup obtained from spherical objects supported on $\kc_l$. 
Then one has
 $B(X)=B(X)|_{\kc_1}\times\dots\times B(X)|_{\kc_r}$
since spherical twists corresponding to fully orthogonal objects commute.
Ishii and Uehara \cite{IU} give a minimal set of $|\kc_l|+1$ generators
\[ B(X)|_{\kc_\ell} = \langle \TTT_{\ko_C(-1)}, \TTT_{\omega_{\kc_\ell}} \mid C\in\Delta_\ell \rangle
                  = \langle \TTT_{\ko_C(-1)}, \TTT_{\ko_C} \mid C\in\Delta_\ell \rangle \]
so that the second set of $2|\kc_l|$ twists also generates; 
here $\Delta_\ell := \{C\in\Delta\mid C\subset\kc_\ell\}$ --- we point out that the results of
\cite{IU} only apply to  chains $\kc_\ell$ of type $A$.
Finally, by \cite[Corollary~37]{IUU}, $B(X)|_{\kc_l}$ is an affine braid group on $|\kc_l|$ strands.
\end{remark}

\section{Exceptional and spherical objects}

\noindent
An object $E\in\DDD(X)$ is \emph{exceptional} if $\Hom^\bullet(E,E)=\kk$, 
i.e.\ it is as simple as possible from the point of view of the derived category.
Toric and, more generally, rational surfaces carry many exceptional objects --- 
enough to form full exceptional collections, see \cite{HP}.
On the contrary, spherical objects are rarer on the surfaces under study because 
they have to be supported on configurations of $-2$-curves.
Therefore, it seems natural to wonder whether spherical objects can be expressed via 
exceptional ones. 

Before stating the theorem, we recall that an \emph{exceptional pair} consists of two
exceptional objects $E',E\in\DDD(X)$ such that $\Hom^\bullet(E,E')=0$. We will call
$(E',E)$ a \emph{special exceptional pair} if it is an exceptional pair with
$\Hom^\bullet(E',E)=\kk\oplus\kk[-1]$.

\begin{theorem} \label{thm:sphexc}
Let $X$ be a smooth, projective surface and suppose $E', E, S \in \DDD(X)$ are objects which fit into 
an exact triangle $E'\to E\to S$. %be an exact triangle in $\DDD(X)$, where $X$ is a smooth, projective surface.

\vspace{0.5ex}

\noindent
\begin{tabular}{c p{0.9\textwidth}}
(i)   & If $(E',E)$ is a special exceptional pair, then $\Hom^\bullet(S,S)=\kk\oplus\kk[-2]$. \\
(ii)  & If $E$ is exceptional, $S$ is spherical and $\Hom^\bullet(E,S)=\kk$, then $(E',E)$ is a
       special exceptional pair. \\
(iii) & If $S$ is spherical and $X$ is a rational surface satisfying conditions (1) and (2) of
        Theorem~\ref{thm:AutDX}, then $(E',E)$ can be chosen to be a special exceptional pair.
\end{tabular}
\end{theorem}

Note that by Theorem~\ref{thm:toric}, smooth projective toric surfaces satisfy
the conditions of (iii) for the above theorem.

\begin{proof}
For (i) and (ii), use the following diagram in $\DDD(\kk)$:
\[ \xymatrix{
\Hom^\bullet(E',E')        \ar[r] & \Hom^\bullet(E',E)       \ar[r] & \Hom^\bullet(E',S) \\
\Hom^\bullet(E,E')  \ar[u] \ar[r] & \Hom^\bullet(E,E) \ar[u] \ar[r] & \Hom^\bullet(E,S) \ar[u] \\
\Hom^\bullet(S,E')  \ar[u] \ar[r] & \Hom^\bullet(S,E) \ar[u] \ar[r] & \Hom^\bullet(S,S) \ar[u]
} \]
A diagram chase around this diagram implies (i), and also (ii), using the assumption 
of sphericality on $S$ to invoke Serre duality.

For claim (iii) we use \cite[Proposition 1.6]{IU} which states that the spherical twists
of objects supported on a chain of $-2$-curves act transitively on these spherical objects.
Using this, together with Theorem~\ref{thm:AutDX} we see that, for any spherical object 
$S\in\DDD(X)$, there exists $\varphi\in\Aut(\DDD(X))$ such that $\varphi(S)\cong\ko_C(a)$
for some $C\in\Delta(X)$ and $a\in\IZ$.
(As in Remark~\ref{rmk:splitting_braid_group}, knowing this property for a single 
$A_n$-chain is enough in order to apply it to $X$.)

Since $X$ is assumed to be rational, line bundles are exceptional objects and we get the exceptional
presentation $\ko_X(-C)\to\ko_X\to\ko_C$ for the sheaf $\ko_C$.
We may contract $\ko_C$ in $X$ to obtain a surface with a rational singularity. Choosing 
a smooth curve which goes through the singular point, its strict transform $H$ in $X$ will
have $H.C=1$. Then $\varphi\inv(\ko_X(-C+aH))\to\varphi\inv(\ko_X(aH))\to S$ is a
presentation for $S$ by exceptional objects.

Finally, the assertions that $\Hom^\bullet(E',E)=\kk\oplus\kk[-1]$ and $\Hom^\bullet(E,E')=0$
follow at once from
$\Hom^\bullet(\ko_X,\ko_X(-C))=H^\bullet(\ko_X(-C))=0$ together with
$\Hom^\bullet(\ko_X(-C),\ko_X)=H^\bullet(\ko_X(C))=\kk\oplus\kk[-1]$.
% $H^1(\ko_X(C))=\kk$, $H^2(\ko_X(C))=0$.
\end{proof}

\begin{example}
Part (i) of the theorem states that $S$ satisfies the Ext-condition for spherical
objects. However, it can happen that $S\not\cong S\otimes\omega_X$, and so $S$ is
not spherical. As a specific example, consider $F_2$, the second Hirzebruch surface. 
It contains a (unique) $-2$-curve $C\subset F_2$; hence the object $\ko_C\in\DDD^b(F_2)$
is spherical. Let $\pi\colon X\to F_2$ be the blow-up of $F_2$ in one (of the two)
torus-invariant points on $C$. We denote by $D$ the exceptional curve and by $C'$
again the strict transform of $C$. Thus, $\pi^{-1}(C)=C'+D$.

The functor $\LLL\pi^*\colon\DDD^b(F_2)\to\DDD^b(X)$ is fully faithful, as follows 
from adjunction, the projection formula and $\RRR\pi_*\ko_X=\ko_{F_2}$ (or see 
\cite[Proposition 11.13]{HuyFM}).
Consider the triangle $\ko_{F_2}(-C)\to\ko_{F_2}\to\ko_C$ in $\DDD^b(F_2)$. 
Here, $(\ko_{F_2}(-C),\ko_{F_2})$ is a special exceptional pair. Pulling back this
triangle under $\LLL\pi^*$ yields $\ko_X(-C'-D)\to\ko_X\to\ko_{C'+D}$. 
Since the pullback functor is fully faithful, $(\ko_X(-C'-D),\ko_X)$ is also a special
exceptional pair. We have $\Hom^\bullet(\ko_{C'+D},\ko_{C'+D})=\kk\oplus\kk[-2]$, from
part (i) of Theorem~\ref{thm:sphexc} or from the fully faithfulness of $\LLL\pi^*$.
However, the sheaf $\ko_{C'+D}$ is not invariant under twisting with $\omega_X$: the 
curves $C'$ and $D$ on $X$ are smooth and rational but of self-intersection $-3$ and $-1$,
respectively.
\end{example}

\section{Proof of Theorem \ref{thm:AutDX}}
\label{sec:proof}

\noindent
Before giving an outline of the proof, we recall the assumptions on the surface $X$:

\begin{tabular}{l @{\hspace{0.8em}} p{0.92\textwidth}}
(1) & The anti-canonical bundle is big. \\
(2) & The $-2$-curves on $X$ form disjoint chains of type $A$. \\
(3) & $\Pic(X)\cong\Pic_\Delta(X)\oplus P$ where $P$ is an $\Aut(X)$-invariant complement.
\end{tabular}

\noindent
Let $\kc=\bigcup_{C\in\Delta}C$ be the union of all $-2$-curves on $X$ and let
$\kc=\kc_1\sqcup\dots\sqcup\kc_r$ be its partition into connected components.
By assumption (2), each $\kc_i$ is a chain of type $A$. Given any autoequivalence
$\varphi\in\Aut(\DDD(X))$, we modify it in three steps until we arrive at a standard
autoequivalence --- for this, we only need conditions (1) and (2):

\noindent
\begin{tabular}{ @{} p{0.1\textwidth} @{} p{0.9\textwidth}}
Step 1: & Modify $\varphi$ using $\Aut(X)$ and $\IZ[1]$ such that points outside of $\kc$ are fixed. \\
%Step 1: & Modify $\varphi$ using an automorphism of $X$ and a shift such that
%          points outside of $\kc$ are fixed. \\
%Step 2: & Show preservation of the subcategories of objects supported on chains $\kc_i$. \\
Step 2: & Show that the subcategory of objects supported on a chain $\kc_i$ is preserved. \\
%Step 2: & The autoequivalence preserves the subcategories $\DDD_{\kc_i}(X)$ of objects with support in $\kc_i$. \\
%Step 2: & For $A\in\DDD(X)$ with $\supp(A)\subset\kc_i$, 
%Step 2: & The equivalence preserves supports inside $\kc_i$. \\
%Step 2: & Show that then objects supported on chains $\kc_i$ are preserved. \\
Step 3: & Invoke the braid group action of Ishii and Uehara \cite{IU} to modify further
          by spherical twists, until all points are fixed.
\end{tabular}
At this stage, the resulting autoequivalence is standard, i.e.\ an element of $A(X)$. 
This proves that $\Aut(\DDD(X))$ is generated by $A(X)$ and $B(X)$. 
Finally we address the relations. It is here that we make use of condition (3):

\noindent
\begin{tabular}{ @{} p{0.1\textwidth} @{} p{0.9\textwidth}}
Step 4: & Prove the decomposition $\Aut(\DDD(X)) = B(X) \rtimes (P \rtimes \Aut(X)) \times \IZ[1]$.
\end{tabular}

\medskip
\noindent
\emph{Step 1.}
By a well-known result of Orlov \cite{Orlov} there is a unique Fourier-Mukai kernel
$P\in\DDD(X\times X)$, so that $\varphi\cong\FM_P$. As the anti-canonical sheaf is big
by assumption, the conditions required for \cite[Theorem~2.3(2)]{Kawamata-DK} hold. In
particular, looking at the proof of this theorem we see that there exists an irreducible
component $Z\subset\supp(P) \subset X\times X$ such that the restrictions to $Z$ of the
natural projections $\pi_1,\pi_2\colon X\times X \to X$ are surjective and birational.
Also see \cite[\S6]{HuyFM} for this.
Following \cite{IU}, we set 
 \[ q := \pi_2|_Z \circ \pi_1|_Z\inv\colon X \dashrightarrow X .\]
As $X$ is a smooth surface, using \cite[Lemma~4.2]{Kawamata-DK}, we note that this 
birational map is a genuine isomorphism --- any birational map between smooth surfaces
is a sequence of blow ups and blow downs but Kawamata's lemma shows that the birational
map in question is an isomorphism in codimension 1. 

Now we show that for any point $x \in X$, the support of $\varphi(k(x))$ is either the
point $q(x)$, or is a connected subset of $\kc$. Note that $\supp(\varphi(k(x))$ must 
be connected, as the map
 $\Hom_{\DDD(X)}(k(x),k(x)) \to \Hom_{\DDD(X)}(\varphi(k(x)),\varphi(k(x)))$
is bijective. It is a general property of equivalences to commute with Serre functors, 
in particular $\varphi(k(x))=\varphi(k(x)\otimes\omega_X)\cong\varphi(k(x))\otimes\omega_X$,
for any point $x$. As $\omega_X$ is a non-trivial line bundle, $\varphi(k(x))$ must
have proper support, i.e.\ $\dim\varphi(k(x))<2$. Therefore $\varphi(k(x))$ is either
supported at a point, or it is supported on a union of curves.
Suppose $C\subset\supp(\varphi(k(x)))$ is any irreducible curve contained in the support.
Since
 $\omega_X|_C\otimes\varphi(k(x))|_C = (\omega_X\otimes\varphi(k(x)))|_C = \varphi(k(x))|_C$
and $\varphi(k(x))$ is supported on the curve $C$, we get $\omega_X|_C=\ko_C$. Hence,
$C\subset X$ is a curve with $K_X.C=0$. Since $-K_X$ is big, it follows from 
Lemma~\ref{lem:zariski} below that $C$ is a smooth, rational curve with $C^2=-2$.
Now looking at the FM transform at the level of its support, we observe that 
\[ q(x) = \pi_2(Z\cap(\{x\}\times X)) \subseteq \pi_2(\supp(P)\cap(\{x\}\times X)) 
        = \supp(\varphi(k(x))). \]
If $\varphi(k(x))$ is supported at a point then this point must be $q(x)$. Otherwise
we have shown that all components of $\supp(\varphi(k(x)))$ are $-2$-curves and so 
$q(x)$ is contained in some $-2$-curve $C$. As $q$ is a surface automorphism, we find
$x\in q\inv(C)$, another $-2$-curve. 
%At this point, we compare the FM transform $\FM_\kp(k(x))$ with the transformation on the cycle level:
%$\supp(\FM_\kp(k(x)))=\pi_2(\supp(P)\cap({x}\times X))$. Using the fact that $Z\subset \supp(P)$ and the above result that
%all components of $\supp(\varphi(k(x))$ are $(-2)$-curves, we find that $q(x)=\pi_2(Z\cap({x}\times X))$ is
%contained in a $-2$-curve $C$ (unless $\varphi(k(x))$ is supported on a point, which then has to be $q(x)$).
%
%As $q$ is a surface automorphism, we find $x\in q\inv(C)$, another $-2$-curve.
%If we write $\kc$ for the union of all $-2$-curves in $X$, then we get: 
%
In particular this implies that if $x\in X\setminus\kc$ then $\varphi(k(x))$ is
supported at the point $q(x)$ and is therefore a shifted skyscraper sheaf of 
length 1
\[ \varphi(k(x))=k(q(x))[i]=q_*(k(x))[i]. \]
The integer $i$ does not depend on $x$: for an equivalence between derived 
categories of smooth, projective schemes, mapping a skyscraper sheaf to a 
skyscraper sheaf is an open property; see \cite[Corollary 6.14]{HuyFM}.
Hence, $\psi:=q^*\circ\varphi[-i]$ is an autoequivalence of $\DDD(X)$ which
fixes all skyscraper sheaves $k(x)$ for $x\in X\setminus\kc$.

\medskip
\noindent
\emph{Step 2.}
We claim that $\psi$ preserves $\kc$, i.e.\ induces an autoequivalence of 
$\DDD_\kc(X)$. Here, $\DDD_\kc(X)$ is the full subcategory of $\DDD(X)$ consisting
of objects whose support is contained in $\kc$. In order to prove the claim,
suppose that $A\in\DDD_\kc(X)$. We need to show that $\supp(\psi(A)) \subseteq \kc$.
If there was $y\in\supp(\psi(A))$, $y\notin\kc$, there would be a non-zero morphism
$\psi(A)\to k(y)$. However, this would imply a non-zero map $A\to\psi\inv(k(y))=k(y)$,
in contradiction to the assumption $\supp(A)\subset\kc$.

In fact we can see that $\psi$ preserves each connected component $\kc_i$. For
this, consider a curve $B$ whose self-intersection number is not $-2$; in 
particular, $B$ is not contained in $\kc$. If $B$ does not meet the component
$\kc_i$, then the same is true for the transform, i.e.\ $\supp(\psi(\ko_B))$
does not intersect $\kc_i$, using same reasoning as in the previous paragraph.
More generally, if $B$ does not meet several of the components, then the same
will be true for the transform. So if we can find enough curves $B$ to separate
the components of $\kc$, then $\psi$ has to preserve each of them. See
Lemma~\ref{lem:separation_general} below for a proof of this fact.

% in the general case and the proof of Lemma~\ref{lem:chains} for a self-contained
% argument in the toric case.

Therefore we are in a position to use the `Key Proposition' of Ishii and Uehara
\cite{IU} repeatedly on each chain of $-2$-curves: there exist an integer $j$ 
and an auto\-equivalence $\Psi \in B(X)$ such that $\Psi \circ \psi$ sends every 
skyscraper sheaf $k(x)$ for $x \in \kc$ to $k(y)[j]$ for some $y \in \kc$. In
\cite{IU}, only globally defined autoequivalences are used, so that the presence
of several chains does not pose an obstacle.

\medskip
\noindent
\emph{Step 3.}
A well-known lemma of Bridgeland and Maciocia (\cite[3.3]{BM}, see also
\cite[Corollary 5.23]{HuyFM}) states that an autoequivalence permuting skyscraper
sheaves of length 1 must be in $\Pic(X)\rtimes\Aut(X)$. Thus we get
\[ \Psi [-j] \circ \psi = \Psi \circ q^*\circ\varphi[-i-j] \in \Pic(X) \rtimes \Aut(X) .\]
Hence $\Aut(\DDD(X))$ is indeed generated by $\Aut(X)$, $\Pic(X)$, $B(X)$ and $\IZ[1]$.

\medskip
\noindent
\emph{Step 4.}
The relations $\Aut(X)\cap B(X)=1$ and $\Pic(X)\cap B(X) = \Pic_\Delta(X)$ are
proved in Lemma~4.14 and Proposition~4.18 of \cite{IU}; note that we can treat
each chain individually using Remark~\ref{rmk:splitting_braid_group}. 

Now we assume that the embedding $\Pic_\Delta(X)\subset\Pic(X)$ splits and that there
is a complement $P$ fixed by $\Aut(X)$ --- this is condition (3) of Theorem~\ref{thm:AutDX}.
We get
\begin{align*}
  A(X) &=     \IZ[1] \times (\Pic(X)\rtimes\Aut(X)) \\
       &\cong \IZ[1] \times \big( (\Pic_\Delta(X)\rtimes\Aut(X)) \oplus (P\rtimes\Aut(X)) \big)
\end{align*}
We thus have two subgroups of $\Aut(\DDD(X))$, namely $\IZ[1] \times (P\rtimes\Aut(X))$
and the normal subgroup $B(X)$, which together generate $\Aut(\DDD(X))$ and whose 
intersection is trivial. Hence we obtain the desired semi-direct product decomposition,
and the proof of Theorem~\ref{thm:AutDX} is finished, apart from the following lemmas.

\begin{lemma} \label{lem:zariski}
Let $X$ be a smooth, projective surface with $-K_X$ big. If $C$ is an irreducible, reduced
curve on $X$ with $K_X.C=0$, then $C$ is a $-2$-curve, i.e.\ smooth and rational with $C^2=-2$.
\end{lemma}

\begin{proof}
A big divisor is pseudo-effective \cite[Theorem~2.2.26]{Lazarsfeld}. Hence we can 
use Zariski's decomposition $-K_X=P+N$, where $P$ and $N$ are $\IQ$-divisors with $P$ nef,
$N$ effective and where $P$ has zero intersection number with every prime divisor of $N$; furthermore $N$ is also
negative definite \cite[Theorem~2.3.19]{Lazarsfeld}.
%
%$N=\sum_i \nu_i N_i$ effective ($N_i$ integral prime divisors) and negative definite (i.e.\
%the matrix of intersections $(N_i.N_j)$ is negative definite) and $P.N_i=0$ for all $i$
%(see \cite[Theorem~2.3.19]{Lazarsfeld}). The decomposition is uniquely determined. 
%
The positive part $P$ carries all the sections of $-K_X$ and is therefore big as well
\cite[Proposition~2.3.21]{Lazarsfeld}. Since $P$ is big and nef, we get $P^2>0$ 
\cite[Theorem~2.2.16]{Lazarsfeld}.

Our next claim is that $K_X.C=0$ implies $P.C=0$: If $C$ is a component of $N$, this is
obvious from the Zariski decomposition. Otherwise we have $C.N\geq0$ as $N$ is
effective. We also find $P.C\geq0$ as $P$ is nef. From $0=(-K_X).C=(P+N).C\geq0$ we
deduce $P.C=0$.

The Hodge index theorem yields $C^2<0$ since $P^2>0$ and $P.C=0$. Finally, applying 
the adjunction formula with $K_X.C=0$ and $C^2<0$ gives $\deg(K_C)=(K_X+C).C=C^2<0$.
Riemann-Roch and duality imply $g(C)=1-\chi(\ko_C)=1+\chi(\omega_C)=1+\deg(K_C)/2\leq0$,
hence $g(C)=0$. It follows that $C$ is rational and smooth ---
see \cite[\S II.11]{BHPvdV}) for details.
Using the adjunction formula again shows $C^2=-2$.
\end{proof}

\begin{lemma} \label{lem:separation_general}
Let $X$ be a projective surface such that all $-2$-curves appear in
$ADE$-chains. Then for any two such chains, there exists a curve meeting one 
chain transversally and avoiding the other.
\end{lemma}

\begin{proof}
Fix two different chains $\kc$, $\kc'$ of $-2$-curves. By assumption, these
are disjoint. We contract $\kc$ and $\kc'$ to obtain a surface $Y$ with two
rational singularities $y$, $y'$. This is possible, i.e.\ $Y$ is algebraic,
since we are dealing with chains of $-2$-curves of type $ADE$; see
\cite[Theorem 2.7]{Artin}. % For the main theorem, only type $A$ is relevant.
In fact, $Y$ is projective since $X$ was. Choosing an ample divisor of 
sufficiently large degree, we find a curve $B\subset Y$ going through $y$
but missing $y'$. Its preimage under the contraction $X\to Y$ then has the
desired property.
\end{proof}

\section{Toric surfaces and proof of Theorem \ref{thm:toric}}
\label{sec:toric}

\noindent
In this section, we will work with a smooth, projective toric surface $X$. 
We start by fixing some notation and gathering a few well-known properties
of toric surfaces that we will use later. As a general reference for the
exposition below, we refer the reader to \cite{Fulton}.

Let $N$ be a rank $2$ lattice and define $N_\IR := N \otimes_\IZ \IR$. A toric
surface $X$ is specified by a fan $\Sigma$ of (strongly convex rational 
polyhedral) cones in $N_\IR$. We denote by $\Sigma(1)$ the set of rays (one
dimensional cones) in $\Sigma$, by $\{v_i\}_{i \in\Sigma(1)} \subset N$ the
set of primitive generators of the rays and by $\{D_i\}_{i \in\Sigma(1)}$ the
set of torus invariant divisors corresponding to the rays; each $D_i$ is an
irreducible, torus-invariant curve.

We assume that the fan is complete (the support of $\Sigma$ is $N_\IR$) which
(in the surface case) is equivalent to the property that $X$ is projective. 
The variety $X$ is smooth and this is equivalent to the condition on the fan,
that for each two-dimensional cone $\sigma$, the generators of the rays of 
$\sigma$ form a basis for $N$. Ordering the generators cyclicly, it follows
that
\[ \alpha_i v_i = v_{i-1} + v_{i+1} \qquad \forall i =1, \dots ,|\Sigma(1)| \]
for some integers $\alpha_i$. It can be shown that $-\alpha_i$ is the 
self-intersection number of $D_i$ for each $i \in\Sigma(1)$. 
Since $X$ is smooth, there is an exact sequence \cite[\S3.4]{Fulton}
\begin{equation}\label{exactseq} 
 0\to M \to \IZ^{\Sigma(1)} \to \Pic(X) \to0,
\end{equation} 
where we denote by $M:=N^\vee$ the dual lattice of $N$. $\Pic(X)$ is a free
abelian group, so $\Pic_\Delta(X)$ is the free abelian subgroup generated by
$\Delta(X)$.

%%%%%%%%%%

\begin{lemma} \label{lem:chains}
$\Delta(X)$ consists of a finite number of chains of type $A$.
\end{lemma}

\begin{proof}
Let $C$ be a curve in $\Delta(X)$. Using the exact sequence~(\ref{exactseq})
we observe that $C$ is linearly equivalent to a sum 
$\sum_{i \in \Sigma(1)} a_i D_i$ of torus invariant divisors indexed by the
rays in the fan $\Sigma$ of $X$. Since $C$ is effective, we may choose 
this Weil divisor in such a way that it is also effective, so $a_i \geq 0$
for each $i\in \Sigma(1)$.
Then 
\[ -2 = C.C = C. \Big( \sum_i a_i D_i \Big) = \sum_i a_i (C.D_i) \]
so there exists some $i\in \Sigma(1)$ such that $C.D_i<0$. Since $C$ and
$D_i$ are both irreducible curves, we conclude that $C=D_i$. Thus all 
curves in $\Delta(X)$ are torus invariant curves corresponding to rays of
$\Sigma$.
Such curves intersect if and only if the corresponding rays span a cone
(see for example \cite[\S5.1]{Fulton}). 
By looking at the fan $\Sigma$ which is supported on $N_\IR \cong \IR^2$
we see that the only possible configurations are a finite number of 
chains of type $A$ or a single closed chain of type 
$\tilde{A}_{|\Sigma(1)|}$.

In order to see that this final possibility doesn't occur, note that if
$D_i^2=-2$ then $2v_i=v_{i-1}+v_{i+1}$ which in turn means that $v_i$ lies
on the line in $N_\IR$ through $v_{i-1}$ and $v_{i+1}$. It is clear however, that 
the generators of the rays of a complete fan can not all be collinear.
%Therefore, an 
%$A_n$-chain of $-2$-curves $D_r,\dots,D_{r+n-1}$ corresponds to the
%generators $v_r,\dots,v_{r+n-1}$ all lying on a straight line in $N_\IR$
%through $v_{r-1}$ and $v_{r+n}$. It is easy to check that the fan $\Sigma$
%must contain additional rays apart from those contained in two such chains. 
%This argument also shows that for each pair of chains of $-2$-curves
%$\kc_1$ and $ \kc_2$, there exists a curve (not contained in $\Delta(X)$)
%which intersects  $\kc_1$ but not $\kc_2$.
\end{proof}

%%%%%%%%%%

\begin{lemma} \label{lem:big}
If $X$ is a smooth, projective, toric variety (not necessarily a surface),
then $-K_X$ is big.
\end{lemma}

\begin{proof}
As is well-known (see \cite[\S4.3]{Fulton}), $-K_X$ is linearly equivalent
to $\sum_{i\in\Sigma(1)}D_i$, the sum  of all torus invariant prime divisors.
Picking an ample divisor $H=\sum_i a_iD_i$, we can assume that all $a_i>0$.
Then $H+mK_X$ is effective for some $m>0$, or in other words, $-mK_X$ is the
sum of an ample and an effective divisor, hence big.
%In particular, the divisor class of $-K_X$ lies in the interior of the
%pseudo-effective cone (since this is generated as a cone by the classes of
%the torus invariant divisors). It follows that $-K_X$ is big, see for example
%\cite[Theorem 2.2.26]{Lazarsfeld}).
\end{proof}

%%%%%%%%%%

\begin{lemma} \label{lem:primitive}
If $X$ is a smooth, projective toric surface containing two divisors 
$D_i$, $D_{i+1}$ corresponding to adjacent rays $i,i+1\in\Sigma(1)$ 
with $D_i^2\neq-2$ and $D_{i+1}^2\neq-2$, then the group embedding 
$\Pic_\Delta(X)\subset\Pic(X)$ splits.
\end{lemma}

\begin{proof}
We use the standard exact sequence
 $0\to M \xrightarrow{\iota} \IZ^{\Sigma(1)} \xrightarrow{\pi} \Pic(X) \to0$.
Since $X$ is smooth, the generators  $v_i$ and $v_{i+1}$ of the rays
$i,i+1\in\Sigma(1)$ form a basis of $N\cong\IZ^2$. Using the dual basis
for $M$ and considering the map $\iota$, it is easy to see that the free
abelian group $\Pic(X)$ has a basis 
$\{\pi(D_j) \mid j\in\Sigma, j\neq i,i+1\}$.
Furthermore, since $D_i^2\neq-2$ and $D_{i+1}^2\neq-2$, the subgroup
spanned by classes of $-2$-curves is generated by elements of this 
basis, and so is primitive in $\Pic(X)$. Hence, the quotient 
$\Pic(X)/\Pic_\Delta(X)$ is free and there exists a splitting.
\end{proof}

%%%%%%%%%%

\begin{example} \label{ex:toric_counterexamples}
We now give an example of a smooth, toric surface $X$ such that the
embedding $\Pic_\Delta(X)\subset\Pic(X)$ of abelian groups does not split. 
Consider the toric surface given by the fan in the following picture:
%\begin{wrapfigure}{r}{0.3\textwidth}
\begin{center}
  \psfrag{1}{\scriptsize $D_1$}
  \psfrag{2}{\scriptsize $D_2$}
  \psfrag{3}{\scriptsize $D_3$}
  \psfrag{4}{\scriptsize $D_4$}
  \psfrag{5}{\scriptsize $D_5$}
  \psfrag{6}{\scriptsize $D_6$}
  \psfrag{7}{\scriptsize $D_7$}
  \psfrag{8}{\scriptsize $D_8$}
  \includegraphics[width=0.2\textwidth]{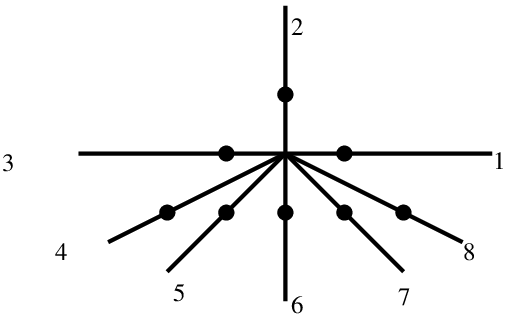}
\end{center}
%\end{wrapfigure}
It can be obtained by blowing up $\IP^1\times\IP^1$ four times. 
The self-intersection numbers are $D_2^2=0$, $D_4^2=D_8^2=-1$, and 
$D_1^2=D_3^2=D_5^2=D_6^2=D_7^2=-2$. 
Choosing $v_1$ and $v_2$ as a basis of $N$, we see that the map 
$M\to\IZ^{\Sigma(1)}$ is given by the transpose of the matrix 
$\bigl(\begin{smallmatrix}
1 & \; 0 & -1            & -2 & -1 & \phantom{-} 0 & \phantom{-} 1 & \phantom{-} 2 \\
0 & \; 1 & \phantom{-} 0 & -1 & -1 & -1            & -1            & -1
\end{smallmatrix}\bigr).$
In particular, the classes of $D_3,\dots,D_8$ form a basis of 
$\Pic(X)$. All of these are $-2$-classes except for $D_4$ and $D_8$.
Writing the $-2$-curve $D_1$ in terms of this basis, we have 
$D_1=2(D_8-D_4)+D_3+D_5+D_7$ in $\Pic(X)$. Therefore $0=2(D_8-D_4)$
in $\Pic(X)/\Pic_\Delta(X)$, so there is torsion. This implies that
the embedding of $\Pic_\Delta(X)$ into $\Pic(X)$ is not primitive.

In fact, it is an easy combinatorial exercise to show that there are only three
smooth, projective toric surfaces which do not have such a splitting. They are given 
by the smooth fans over the following polygons --- here and in the following, the 
vertices on the boundary of the polygon are the generators of the rays of the fan.
Circular dots (\includegraphics[scale=.4]{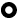}) indicate $-2$-curves.

{\centerline{
  \includegraphics[height=7.5ex]{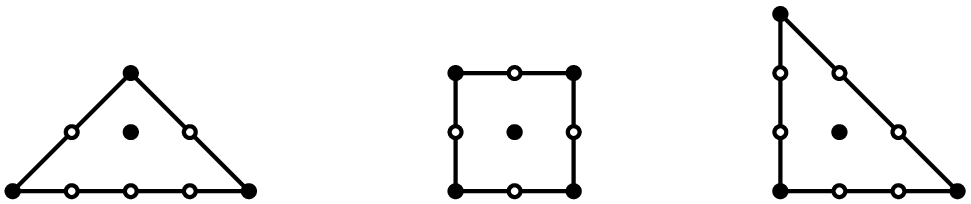}
}}
\end{example}

\begin{lemma} \label{lem:toric_aut}
If $X$ is not one of the three surfaces in Example~\ref{ex:toric_counterexamples},
then there exists an $\Aut(X)$-invariant complement $P$ for the subgroup
$\Pic_\Delta(X)$ in $\Pic(X)$ if and only if there exists an $\Aut(\Sigma(X))$-invariant
complement.
\end{lemma}

\begin{proof}
This follows at once from two geometric facts about toric varieties:

First, $\Aut(X)$ is generated by its identity component $\Aut^0(X)$ and the subgroup
$\Aut(\Sigma(X))$ of fan automorphisms (the latter is by definition the subgroup of 
lattice automorphisms of $N$ fixing the fan $\Sigma$). This statement is a corollary
of Demazure's Structure Theorem \cite[\S3.4]{Oda}.

Second, $\Aut^0(X)$ acts trivially on all of $\Pic(X)$ because the Picard group of a 
toric variety is discrete, i.e.\ $\Pic^0(X)=0$.
\end{proof}

Together, Lemmas~\ref{lem:chains}, \ref{lem:primitive} and \ref{lem:toric_aut} 
prove all parts of Theorem~\ref{thm:toric}. It remains to investigate when an 
$\Aut(\Sigma(X))$-invariant complement exists. For trivial reasons, this is always
true if $\Aut(\Sigma(X))=1$. For more symmetric toric surfaces, both answers are 
possible, as the next two examples show.

\begin{example}
Suppose $\Aut(\Sigma(X))=\IZ/2$ and the action exchanges two rays which do not 
correspond to $-2$-curves, and whose generators form a $\IZ$-basis for $N$. For
example the toric surfaces given by fans over the following polygons:

\medskip

{\centerline{
  \includegraphics[height=10ex]{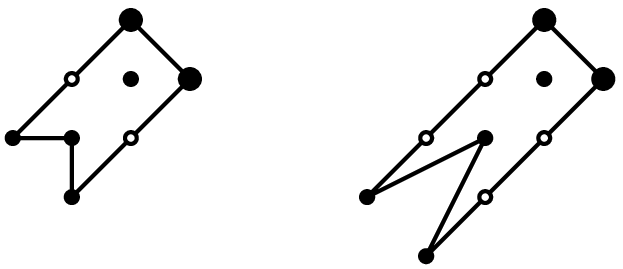}
}}

Excluding the two marked curves (\includegraphics[scale=0.3]{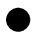}), the remaining torus invariant divisors form a basis 
for $\Pic(X)$, and the subset of these divisors which are not $-2$-curves generate an 
$\Aut(\Sigma(X))$-invariant complement to $\Pic_\Delta(X)$.
Similarly, and again in the case $\Aut(\Sigma(X))=\IZ/2$, suppose there exists a 
$\IZ$-basis for $N$ coming from a ray which is fixed by the action and has odd 
self-intersection number, and another ray which doesn't correspond to a $-2$-curve.
For example, consider fans over the following polygons, where the basis
for $N$ is again marked:

\medskip

{\centerline{
  \includegraphics[height=10ex]{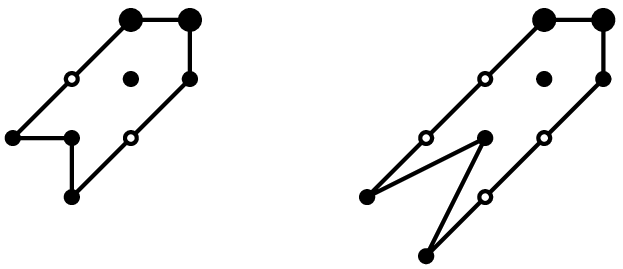}
}}

It is then possible to show that there exists an invariant linear combination of the 
fixed divisor and the two divisors in the $\Aut(\Sigma(X))$-orbit of the non-fixed marked
ray, which, together with all the remaining torus invariant divisors, forms a basis for 
$\Pic(X)$. Again, the subset of these divisors which are not $-2$-curves generate an 
$\Aut(\Sigma(X))$-invariant complement to $\Pic_\Delta(X)$.
%They form the start of two series of toric surfaces with the only non-trivial fan
%automorphism being a reflection. For all of them, an invariant complement exists.
\end{example}

\begin{example} \label{computerexample}
For the following example, computer algebra was used to make sure that no
invariant complement exists. Note that the rays fixed by the $\Aut(\Sigma(X))$ action correspond 
to curves with even self-intersection number, so the argument in the previous 
example doesn't apply.

{\centerline{
  \includegraphics[height=5ex]{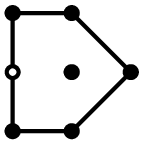}
}}

\end{example}

We conclude this section with a few general observations about, and on the 
construction of, some classes of examples: As a straightforward consequence 
of Theorem~\ref{thm:AutDX}, we see that any smooth projective toric surface
without $-2$-curves has no autoequivalences beyond the standard ones, i.e.\
$\Aut(\DDD(X))=A(X)$. We note that there are infinitely many examples of
such surfaces including, for example, all Hirzebruch surfaces $F_n$ for
$n>2$. It is not hard to check that $-K_X$ is ample if and only if there 
are no torus invariant curves of self-intersection $-2$ or lower. In fact,
there are famously just five smooth toric Fano surfaces \cite[Proposition 2.21]{Oda}.
Therefore, there are infinitely many smooth projective toric surfaces where $-K_X$ is not
ample (and so are not covered by the theorem of Bondal and Orlov \cite{BO})
but for which $\Aut(\DDD(X))=A(X)$. 

On the other hand, it is easy to construct examples with more interesting 
groups of autoequivalences. If $v_0,v_1$ form a basis for a rank two lattice 
$N$, then we can define inductively $v_{s+1}=2v_s-v_{s-1}$ for 
$s=1,\dots,\ell$. Taking these as generators of rays of a fan, we can
produce a complete smooth fan by adding extra rays (with generators 
$v_{\ell+2},\dots,v_{d-1}$) making sure we do not subdivide any of the
existing cones. This doesn't affect the self-intersection numbers of 
$D_{1},\dots,D_{\ell}$, which are by construction $-2$. Indeed by making
an appropriate choice of $v_{\ell+2},\dots,v_{d-1}$, we can ensure that
$D_{0}$ and $ D_{\ell+1}$ do not have self-intersection number $-2$. 
Therefore it is possible to construct a smooth projective toric surface
with a chain of $-2$-curves of arbitary length. Blowing up the intersection
point of two torus invariant curves $D_{s}$ and $ D_{s+1}$ (which 
corresponds to subdividing the cone spanned by $v_{s}$ and $ v_{s+1}$) has 
the effect of reducing the self-intersection numbers of the strict 
transforms $\widetilde{D}_{s}$ and $\widetilde{D}_{s+1}$ by $1$, and
inserting an exceptional $-1$-curve. In this way we can split up a chain
of $-2$-curves into pieces. In fact, we can produce any number of chains
of $-2$-curves of any length.

%$v_k := k.v_1 + (k-1).v_0 \in N$

\section{Surfaces with $\kk^*$-action and proof of Theorem \ref{thm:complexity-one}}
\label{sec:complexity-one}

\noindent
In this final section, we present some non-toric surfaces to which our results
apply. These examples will be certain rational surfaces with $\kk^*$-action.
As references on such surfaces we use mainly the classical \cite{OW} and also
\cite{PS}. %,\cite{IS}.

Start with the trivially ruled surface $C\times\IP^1$, where $C$ is
a smooth, projective curve of genus $g$. This surface inherits a
$\kk^*$-action from the natural action on $\IP^1$, and the fixed points
make up the two curves $F^+:=C\times\{0\}$ and $F^-:=C\times\{\infty\}$.

Blowing up a fixed point produces another surface with $\kk^*$-action. The
exceptional divisor consists of fixed points, so that the process can be 
iterated. Likewise, all negative curves consist of fixed points, and can 
thus be contracted to a surface with $\kk^*$-action. By \cite[Theorem 2.5]{OW},
all smooth surfaces obtained in this fashion have the following configuration
of fixed curves, made up of of $r$ arms, where we denote the curves in the
$\ell$-th arm by $E_{\ell,1},E_{\ell,2},\dots$, starting from $F^+$:
\begin{center}
  \psfrag{F+}{$\scriptstyle F^+$}
  \psfrag{F-}{$\scriptstyle F^-$}
  \psfrag{E11}{$\scriptstyle E_{1,1}$}
  \psfrag{E12}{$\scriptstyle E_{1,2}$}
  \psfrag{Er1}{$\scriptstyle E_{r,1}$}
  \psfrag{Er2}{$\scriptstyle E_{r,2}$}
  \includegraphics[width=0.38\textwidth]{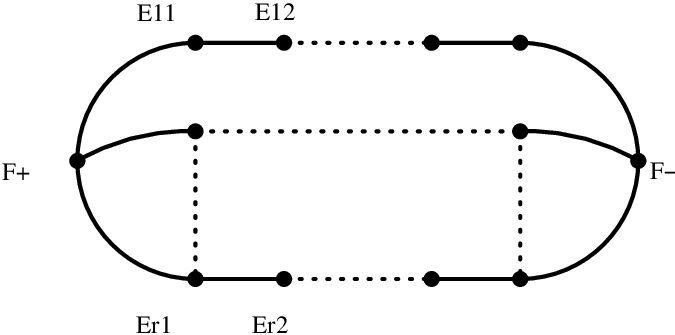}
\end{center}
In fact, all smooth, projective surfaces with an effective $\kk^*$-action
can be obtained in this way, where we allow $F^+$ or $F^-$ to be contracted
in case they are $-1$-curves, \cite[Proposition 2.6]{OW}. Let $X$ be such a
surface with associated graph as above.

By construction, the N\'eron-Severi group of $X$ is generated by the $E_{\ell,i}$,
$F^+$, $F^-$ and $D$, the closure of a generic $\kk^*$-orbit. Thus $D^2=0$,
$D.F^\pm=1$ and $D.E_{\ell,i}=0$. Then, by \cite[Theorem 3.2.1]{PS}, the
anti-canonical divisor has the form
 $-K_X=F^+ + F^- + (2-2g-r)D + \sum_{\ell,i}E_{\ell,i}$.

We need to impose two conditions on $X$: First, all isotropy groups are
connected, i.e.\ we exclude non-zero cyclic groups. Second, the surface
is rational, i.e.\ $g=0$ (but note that $\kk^*$ still only acts on one
factor of the original surface $\IP^1\times\IP^1$).

Rationality implies $\Pic(X)=\text{NS}(X)$. Furthermore, this is a free abelian 
group (see e.g.\ \cite[3.15]{PS}). We proceed to verify the assumptions of
Theorem~\ref{thm:AutDX}.

\begin{lemma}
If all isotropy groups are trivial, then $-2$-curves occur in chains of type $A$.
\end{lemma}

\begin{proof}
\cite[3.5]{OW} describes the isotropy groups from the intersection graph via
continued fractions. The isotropy groups being trivial forces the sequence of 
self-intersection numbers of each arm to be $-1,-2,\dots,-2,-1$. In particular,
$-2$-curves can only occur in chains of type $A$ (note that $F^+$ or $F^-$ can 
also be $-2$-curves).
\end{proof}

\begin{lemma}
Let $X$ be rational with trivial isotropy groups. Then $-K_X$ is big.
\end{lemma}

\begin{proof}
We start by showing that $D=\sum_i E_{\ell,i}$ in $\Pic(X)$, where $\ell$ is 
fixed, i.e.\ the divisor is given by the curves on any arm of the above graph.
The curves intersect as follows:
If $E_{\ell,i}$ meets $F\in\{F^+,F^-\}$, then $E_{\ell,i}.F=1$ and $E_{\ell,i}^2=-1$;
otherwise $E_{\ell,i}.F^\pm=0$ and $E_{\ell,i}^2=-2$. 
Further, $D.F^\pm=E_{\ell,i}.E_{\ell,i+1}=1$ and all other intersection products vanish.
This implies $C.D=C.\sum_iE_{\ell,i}$ for $C$ any of the curves $F^+,F^-,D,E_{k,j}$.
Since those curves generate the Picard group, the divisors $D$ and $\sum_iE_{\ell,i}$
are numerically equivalent. They are then also linearly equivalent as there are no
non-trivial line bundles of degree 0.

Thus we know $-K_X=F^++F^-+2D$. This already implies that $-K_X$ is in the 
pseudo-effective cone of $X$. Furthermore, the relation also shows that all of
$-mK_X-F^-$, $-mK_X-F^+$, $-mK_X-D$ and $-mK_X-E_{\ell,i}$ lie in the pseudo-effective
cone, if $m\gg0$. Hence $-K_X$ sits in the interior of the pseudo-effective cone and 
is therefore big, see \cite[Theorem 2.2.26]{Lazarsfeld}.
\end{proof}

Observe that a toric surface always has big $-K_X$ (see Lemma~\ref{lem:big}) 
whereas blowing up $\IP^2$ in nine general points produces a rational surface on
which $-K_X$ is not big anymore. Surfaces with $\kk^*$-action do not necessarily
have big anti-canonical class and we need to restrict to examples with trivial 
isotropy in order to have this property.

\begin{lemma}
If $X$ is rational with trivial isotropy groups and such that not both $F^+$
and $F^-$ are $-2$-curves, then the inclusion $\Pic_\Delta(X)\subset\Pic(X)$ splits.
\end{lemma}

\begin{proof}
$\Pic(X)$ is the quotient of the free abelian group generated by $F^+$, $F^-$, $D$
and the exceptional curves $E_{\ell,i}$, subject to the relations 
$F^+-F^-+\sum_{\ell,i}(i-1)E_{\ell,i}+(F^-)^2D=0$ and $D=\sum_iE_{\ell,i}$
for all $\ell$ (see \cite[Corollary 3.5]{PS}).

In the quotient $\Pic(X)/\Pic_\Delta(X)$ we observe that $E_{\ell,i}=0$ for all exceptional curves
not adjacent to $F^+$ or $F^-$, as these are all $-2$-curves. Also, by assumption, at least one of $F^+$ and $F^-$ will
survive in the quotient. 

Therefore, for each arm we have a relation in which the two remaining classes have coefficient 1, 
and one further relation in which the classes of $F^+$ and $F^-$ have coefficient 1.
It follows that the quotient $\Pic(X)/\Pic_\Delta(X)$ has a basis consisting
of the curves $E_{\ell,1}$ for all $\ell$, together with $D$ and either $F^+$ or $F^-$
(or neither if one of them is a $-2$-curve). 
\end{proof}

\begin{remark} \label{rem:picard6}
The numbers $(F^+)^2$ and $(F^-)^2$ are not arbitrary: assuming trivial 
isotropy, they can attain any values subject to the restriction 
$(F^+)^2+(F^-)^2=2-\rank(\Pic(X))$; see \cite[Theorem 2.5(iv)]{OW}. In particular, 
$(F^+)^2=(F^-)^2=-2$ forces $\rank(\Pic(X))=6$.
\end{remark}

We give examples of surfaces with $\kk^*$ which meet all conditions of
Theorem~\ref{thm:AutDX}:

\begin{lemma}
Let $X$ be rational with trivial isotropy groups and either $(F^+)^2<-2$
or $(F^-)^2<-2$. If all arms of the intersection graph have different lengths,
then $\Aut(X)$ fixes $\Pic(X)$ element-wise.
\end{lemma}

\begin{proof}
Without loss of generality, we may assume $(F^+)^2<-2$. Since $F^+$ is a 
negative curve of minimal intersection number, it is fixed by all automorphisms.
By the assumption on arm lengths, all other negative curves are also fixed.
The remaining curve $F^-$ is then likewise fixed.
\end{proof}

We finish with a comment on the relationship between the two types of
examples: a surface with $\kk^*$-action as presented here will be toric
(i.e.\ admit an action of $(\kk^*)^2$ compatible with the original action)
only if $r\leq2$, cf.\ \cite[\S4.2]{OW} --- this leads to a circular
intersection graph corresponding to the rays in the fan of a toric surface.

The first toric surface of Example~\ref{ex:toric_counterexamples} is a surface 
with $\kk^*$-action which has no cyclic isotropy groups.
This surface has $r=1$ and both $F^-$ and $F^+$ are $-2$-curves. The divisor $D_2$ of
that example is the closure $D$ of a generic $\kk^*$-orbit mentioned in this section.
By contrast, the surface given by the square polygon of
Example~\ref{ex:toric_counterexamples} comes from a $\kk^*$-surface with $r=2$. 
By Remark~\ref{rem:picard6}, the third example does not lead to a surface with 
$\kk^*$-action of trivial isotropy, as it has Picard rank 7.

\bigskip
\noindent
N. Broomhead, D. Ploog: Institut f\"ur Algebraische Geometrie, Leibniz Universit\"at Hannover,
Welfengarten 1, 30167 Hannover, Germany.

\noindent
e-mail: broomhead@math.uni-hannover.de; ploog@math.uni-hannover.de
\end{document}